\documentclass{amsart}
\usepackage{amsfonts}
\usepackage{amsmath,amssymb}
\usepackage{amsthm}
\usepackage{amscd}
\usepackage{graphics}
\usepackage{graphicx}

\theoremstyle{remark}{
\newtheorem{Def}{{\rm Definition}}

\newtheorem{Prob}{{\rm Problem}}

}
\theoremstyle{plain}
{

\newtheorem{Prop}{Proposition}

\newtheorem{MainThm}{Main Theorem}

}

\begin{document}
\title[Hyperbolas and related new real algebraic maps]{On real algebraic maps whose images are domains surrounded by the products of hyperbolas and real affine spaces}
\author{Naoki kitazawa}
\keywords{Real algebraic functions and maps. Real algebraic sets. Real algebraic hypersurfaces. Hypersurfaces of degree 1 or 2. \\
\indent {\it \textup{2020} Mathematics Subject Classification}: Primary~14P05, 14P10, 14P25, 57R45, 58C05. Secondary~52C17, 57R19.}
\address{Osaka Central
	Advanced Mathematical Institute, 3-3-138 Sugimoto, Sumiyoshi-ku Osaka 558-8585 \\
	TEL (Office): +81-6-6605-3103 \\
	FAX (Office): +81-6-6605-3104 \\
}
\email{naokikitazawa.formath@gmail.com}
\urladdr{https://naokikitazawa.github.io/NaokiKitazawa.html}

\maketitle
\begin{abstract}
Previously, we have systematically constructed explicit real algebraic functions which are represented as the compositions of smooth real algebraic maps whose images are domains surrounded by hypersurfaces of degree 1 or 2 with canonical projections. Here we give new examples with the hypersurfaces each of which is the product of a connected component of a hyperbola and a copy of an affine space explicitly. As a related future work we also discuss problems to obtain the zero sets of some real polynomials explicitly from increasing sequences of real numbers.

This is motivated by a problem in theory of smooth functions proposed first by Sharko: can we construct nice smooth functions whose Reeb graphs are as prescribed? The Reeb space of a smooth function is the naturally obtained graph whose underlying space is the quotient space of the manifold consisting of connected components of preimages. 
The author first considered variants respecting the topologies of the preimages and obtained several results before.
Our work is also motivated by real algebraic geometry, pioneered by Nash. We can know existence of real algebraic structures of smooth manifolds and some general sets and we already know several approximations of smooth maps by real algebraic maps. Our interest lies in explicit construction, which is difficult.

\end{abstract}
\section{Introduction.}
\label{sec:1}
\subsection{A short presentation on real algebraic manifolds and maps and our motivation for construction of such maps and the manifolds explicitly.}
Real algebraic manifolds and smooth real algebraic maps between such manifolds have been studied as important topics in real algebraic geometry. 

In short, real algebraic cases are based on the zero sets of real polynomials, which are so-called {\it real algebraic} sets {\it defined by the polynomials}. 
Similarly, Nash ones are based on the sets defined by equations or inequalities on real polynomials, which are so-called {\it semi-algebraic} sets. See \cite{shiota} for example. The Nash category is originally considered in an old version \cite{kitazawa6}, which will be withdrawn, of our paper. 

In our paper, fundamental notions from related theory and our presentation on related history are based on \cite{akbulutking, bochnakcosteroy, kollar, kucharz, nash, tognoli} for example.

It is important to know existence of such manifolds and maps.
For example, Nash and Tognoli have shown that smooth closed manifolds are regarded as the real algebraic sets defined by some real polynomials. Akbulut and King with various researchers on real algebraic geometry have contributed to studies on problems asking whether subsets in Euclidean spaces are real algebraic sets, real algebraic manifolds, or approximated by suitable sets of such classes, for example. Approximations of smooth maps by real algebraic maps are also important and some affirmative answers have been known. Some are still open and important problems.


Our interest is on explicit construction of explicit real algebraic functions or maps of non-positive codimensions. Natural projections of spheres embedded naturally in the one-dimensional higher Euclidean spaces are simplest examples on closed manifolds. As functions regarded as generalized cases, some functions on so-called {\it symmetric spaces} are well-known as \cite{ramanujam, takeuchi} show.
They are given by natural real polynomials. Their global structures are hard to understand. For example, can we obtain information on preimages? It seems to be very difficult in general.

As pioneering results on such explicit construction, the author gives explicit examples in \cite{kitazawa3, kitazawa5, kitazawa6, kitazawa7} for example. This is based on a problem asking whether we can construct nice smooth functions with mild singularities and prescribed preimages. This is also a revised version of Sharko's question and first considered by the author essentially in \cite{kitazawa1, kitazawa2}. \cite{sharko} asks whether for a graph we can construct a nice function whose Reeb graph is the graph. Related to these studies, \cite{masumotosaeki, michalak} are also important, for example. The {\it Reeb graph} of a smooth function of a certain nice class such as one in \cite{saeki2} is a natural graph whose underlying space is the space of all connected components of preimages and the canonically defined quotient space of the manifold. \cite{reeb} is one of classical pioneering papers on Reeb graphs.

\subsection{Terminologies and notation on manifolds and maps and our problems and new answers.}
Let $X$ be a topological space which is homeomorphic to a cell complex of a finite dimension. We can define the dimension $\dim X$ as a unique integer. A topological manifold is well-known to be homeomorphic to some CW complex. A smooth manifold is well-known to have the structure of a polyhedron and we have one in a canonical way: the polyhedron is a so-called PL manifold. A topological space homeomorphic to a polyhedron whose dimension is at most $2$ is well-known to have the structure of a polyhedron and it is also unique. For a topological manifold, such a nice fact holds in the case where the dimension is at most $3$. Check the celebrated theory \cite{moise} for example. 
  
Let ${\mathbb{R}}^k$ denote the $k$-dimensional Euclidean space. This is a simplest $k$-dimensional smooth manifold. This is also regarded as the Riemannian manifold endowed with the standard Euclidean metric. Let $\mathbb{R}:={\mathbb{R}}^1$ and $\mathbb{Z} \subset \mathbb{R}$ be the set of all integers. Let $\mathbb{N} \subset \mathbb{Z}$ the set of all integers.
For each point $x \in {\mathbb{R}}^k$, $||x|| \geq 0$ can be defined as the distance between $x$ and the origin $0$ under the metric.
This is also naturally a smooth real algebraic manifold, smooth Nash manifold and a real analytic manifold. For example, the real algebraic manifold is called the k-dimensional real affine space. $S^k:=\{x \in {\mathbb{R}}^{k+1} \mid ||x||=1\}$ denotes the $k$-dimensional unit sphere, which is a $k$-dimensional smooth compact submanifold of ${\mathbb{R}}^{k+1}$ and has no boundary. It is connected for any positive integer $k \geq 1$. It is a discrete set with exactly two points for $k=0$. It is a smooth real algebraic set, which is the zero set of the real polynomial $||x||^2={\Sigma}_{j=1}^{k+1} {x_j}^2$ with  $x:=(x_1,\cdots,x_{k+1})$. 
$D^k:=\{x \in {\mathbb{R}}^{k} \mid ||x|| \leq 1\}$ is the $k$-dimensional unit disk, which is a $k$-dimensional smooth compact and connected submanifold of ${\mathbb{R}}^{k}$ for any non-negative integer $k \geq 0$. 
This is a semi-algebraic set.

Let $c:X \rightarrow Y$ be a differentiable map between two differentiable manifolds $X$ and $Y$. $x \in X$ is a {\it singular} point of the map if the rank of the differential at $x$ is smaller than the minimum between the dimensions $\dim X$ and $\dim Y$. We call $c(x)$ a {\it singular value} of $c$.
Let $S(c)$ denote the set of all singular points of $c$.
In our paper, for differentiable maps, we consider smooth maps, which are maps of the class $C^{\infty}$, unless otherwise stated. Smooth maps between manifolds regarded as smooth manifolds always mean such maps.  
For smooth manifolds, {\it non-singular} real algebraic manifolds are considered as so-called non-singular (real algebraic) manifolds in the theory of (real) algebraic manifolds. More rigorously, we only consider unions of connected components of some real algebraic sets essentially and define such notions by the rank of the map defined canonically from the polynomials. 

A canonical projection of a Euclidean space ${\mathbb{R}}^k$ is the smooth surjective map mapping 
each point $x=(x_1,x_2) \in {\mathbb{R}}^{k_1} \times {\mathbb{R}}^{k_2}={\mathbb{R}}^k$ to the first component $x_1 \in {\mathbb{R}}^{k_1}$ where the conditions on the dimensions are given by $k_1, k_2>0$ and $k=k_1+k_2$. Let ${\pi}_{k,k_1}:{\mathbb{R}}^{k} \rightarrow {\mathbb{R}}^{k_1}$ denote the map. We define a canonical projection of the unit sphere $S^{k-1}$ as the restriction of this canonical projection.

We go back to our motivation, problems and answers.
For $X \subset \mathbb{R}$ and $a \in X$, we define $X_{a}:=\{x \in X \mid x \leq a \}$.
\begin{Prob}
	\label{prob:1}
	Let $\{t_j\}_{j=1}^l$ be an increasing sequence of real numbers of length $l>1$ and $l_{{\mathbb{N}}_{l-1}}:{\mathbb{N}}_{l-1} \rightarrow \{0,1\}$ a map. For a sufficiently large integer $m$, can we have an $m$-dimensional non-singular real algebraic connected manifold $M$ with no boundary and a smooth real algebraic function $f:M \rightarrow \mathbb{R}$ enjoying the following properties.
	\begin{enumerate}
		\item The image is $f(M)=[t_1,t_{l}]$ and the image $f(S(f))$ is $\{t_j\}_{j=1}^l$.
\item The preimage $f^{-1}(t)$ is connected and homeomorphic to a CW complex whose dimension is at most $m-1$ for any $t \in f(M)$.
		\item The preimage $f^{-1}(t)$ is a smooth compact submanifold with no boundary if $t \in (t_{j},t_{j+1})$ and $l_{{\mathbb{N}}_{l-1}}(j)=0$. The preimage $f^{-1}(t)$ is a smooth non-compact submanifold and has no boundary if $t \in (t_{j},t_{j+1})$ and $l_{{\mathbb{N}}_{l-1}}(j)=1$. 
\item $f$ is represented as the composition of a smooth real algebraic map $f_{D}:M \rightarrow {\mathbb{R}}^n$ for a suitable integer $n \geq 3$ enjoying the following nice properties with a canonical projection ${\pi}_{n,1}$.
\begin{enumerate}
\item The image $f_{D}(M)$ is the closure $\overline{D}$ of a nice open set $D \subset {\mathbb{R}}^n$.
\item $D$ is surrounded by smooth real algebraic hypersurfaces being connected components of the real algebraic sets defined by some real polynomials of degrees at most $2$.
\end{enumerate}
	\end{enumerate}
$D$ is formulated rigorously in Definition \ref{def:1} and $f_D$ is presented in Proposition \ref{prop:1} explicitly, later.
	\end{Prob}
We give an explicit affirmative answer in \cite{kitazawa7}. Our preprint \cite{kitazawa6}, which will be withdrawn, considers construction of smooth Nash maps and manifolds where these objects are revised to ones in the real algebraic category.
\begin{MainThm}[This is presented again as Main Theorem \ref{mthm:2} more rigorously]
\label{mthm:1}
We have a new affirmative answer to Problem \ref{prob:1} taking each hypersurface surrounding $D$ as the product of a connected component of a hyperbola and a copy of an affine space whereas in \cite{kitazawa7} they may be other kinds of connected components of the real algebraic sets defined by some real polynomials of degrees at most $2$. 
\end{MainThm}
For this, especially, a smooth case \cite{kitazawa2} has motivated. There we have constructed smooth functions with prescribed preimages consisting of copies of Euclidean spaces and unit spheres on compact or non-compact manifolds with no boundaries. \cite{kitazawa3, kitazawa5} are cases where the manifolds $M$ are non-singular, real algebraic and closed.
The next section explains about this including the proof and related new problems. 

\ \\
\noindent {\bf Conflict of interest.} \\
The author was a member of the project JSPS Grant Number JP17H06128 (Principal Investigator: Osamu Saeki). He was
also a member of the project JSPS KAKENHI Grant Number JP22K18267 (Principal Investigator: Osamu Saeki). Under their support we do this study. The author is a researcher at Osaka Central
Advanced Mathematical Institute (OCAMI researcher) whereas he is not employed there. This also supports our study. \\
\ \\
{\bf Data availability.} \\
Data essentially supporting our present study are all in the
 paper.
\section{On Main Theorems.}

We introduce notation on a hyperbola. 
We consider real numbers  $a$, $b$ and $c \neq 0$.
We consider the subset $C_{{\rm H},a,b,c}:=\{(x_1,x_2) \in {\mathbb{R}}^2 \mid (x_1-a)(x_2-b)=c\}$. 
This is a $1$-dimensional non-singular real algebraic manifold and diffeomorphic to $\mathbb{R} \sqcup \mathbb{R}$.
The two connected components is denoted by $C_{{\rm H}_{+},a,b,c}:=\{(x_1,x_2) \in C_{{\rm H},a,b,c} \mid x_2>b\}$ and $C_{{\rm H}_{-},a,b,c}:=\{(x_1,x_2) \in C_{{\rm H},a,b,c} \mid x_2<b\}$. This is also a so-called {\it hyperbola}.

We can easily guess the definition of a {\it real affine transformation} of a real affine space. 
However, we explain about it shortly.
Of course a real affine space is naturally regarded as a real vector space and we can consider linear isomorphisms (transformations).
An {\it affine transformation} on a given real affine space means a transformation represented as the composition of a linear isomorphism (transformation) with an operation of obtaining the sum of the original element and some fixed element of the real affine space.
\subsection{A proposition for our proof of Main Theorems.}
Our main ingredient is, in short, explicitly reviewing construction of smooth real algebraic maps first presented in \cite{kitazawa3}, followed by \cite{kitazawa5} in a generalized way.
We prepare the following for our proof of Main Theorems. For a finite set $X$, $|X| \in \mathbb{N} \subset \{0\}$ denotes its size.
\begin{Def}
\label{def:1}
Let $D$ be a connected open set in ${\mathbb{R}}^n$ enjoying the following properties. 
\begin{enumerate}
\item
\label{def:1.1}
There exist a suitable positive integer $l>0$ and a family $\{f_j\}_{j=1}^l$ of $l$ real polynomials with $n$ variables.  
\item
\label{def:1.2}
$D:={\bigcap}_{j=1}^l \{x \in {\mathbb{R}}^n \mid f_j(x)>0\}$ holds and it is also a semi-algebraic set. 
\item
\label{def:1.3}
For the closure $\overline{D}$ of $D$, $\overline{D}={\bigcap}_{j=1}^l \{x \in {\mathbb{R}}^n \mid f_j(x) \geq 0\}$ holds and it is also a semi-algebraic set.

\item
\label{def:1.4}
$S_j$ is the union of some of connected components of the real algebraic set $\{x \in {\mathbb{R}}^n \mid f_j(x)=0\}$ and a non-singular real algebraic hypersurface with no boundary. Furthermore, $\{x \in {\mathbb{R}}^n \mid f_j(x)=0\}-S_j$ and the set $\overline{D}$ are disjoint.
\item
\label{def:1.5}
Distinct hypersurfaces in $\{S_j\}$ intersect satisfying the following conditions on the "transversality".
\begin{itemize}
\item Each intersection ${\bigcap}_{j \in \Lambda} S_j$ ($\Lambda \subset {\mathbb{N}}_{l}$) is empty or a non-singular real algebraic hypersurface with no boundary and of dimension $n-|\Lambda|$ (where $\Lambda$ is not empty).
\item Let $e_j:S_j \rightarrow {\mathbb{R}}^n$ denote the canonically defined smooth (real algebraic) embedding. For each intersection ${\bigcap}_{j \in \Lambda} S_j$ before and each point $p$ there, the image of the differential of the canonically defined embedding of ${\bigcap}_{j \in \Lambda} S_j$ at the point $p$ and the intersection of the images of all differentials of the embeddings in ${\{e_j\}}_{j \in \Lambda}$ at the point $p$ coincide. 
\end{itemize}
\end{enumerate}
Then $D$ is said to be a {\it normal and convenient} semi-algebraic set. We also call it an {\it NC} semi-algebraic set. 
\end{Def}
Note that "normal and convenient" or "NC" is identical to our paper (and our preprints \cite{kitazawa6, kitazawa7}), as we know. Note also that some conditions are a bit different from them.

The following reviews main ingredients of \cite{kitazawa3, kitazawa5}. This is also similar to related arguments in \cite{kitazawa6, kitazawa7}.
\begin{Prop}
	\label{prop:1}
For an NC semi-algebraic set $D$ in ${\mathbb{R}}^n$ and an arbitrary sufficiently large integer $m$, we have a non-singular real algebraic connected manifold $M$ which is also a connected component of the real algebraic set defined by some $l$ real polynomials and a smooth real algebraic map $f_D:M \rightarrow {\mathbb{R}}^n$ enjoying the following properties.
\begin{enumerate}
\item The image $f_D(M)$ is the closure $\overline{D}$.
\item $f_D(S(f_D))=\overline{D}-D$.
\item The preimage ${f_D}^{-1}(p)$ of each point $p$ is at most {\rm (}$m-n${\rm )}-dimensional and it is seen as a smooth submanifold diffeomorphic to the product of finitely many copies of unit spheres if $p \in D$.  
\end{enumerate} 
\end{Prop}

\begin{proof}

This reviews a main result or Main Theorem 1 of \cite{kitazawa5} explicitly. This is also similar to original proofs \cite{kitazawa6, kitazawa7}. For example, we abuse the notation from Definition \ref{def:1}.

We use the notation like $x:=(x_1,\cdots,x_k)$ for an arbitrary integer $k>0$ in considering local coordinates and explicit points for example.

We define a subset $S \subset {\mathbb{R}}^{m+1}$ by
$S:= \{(x,y) \in \overline{D} \times {\mathbb{R}}^{m-n+1} \subset {\mathbb{R}}^n \times {\mathbb{R}}^{m-n+1}={\mathbb{R}}^{m+1} \mid f_j(x_1,\cdots,x_n)-||y_j||^2=0, j \in {\mathbb{N}_l}\}$ and investigate this. 
First, this is a semi-algebraic set in ${\mathbb{R}}^{m+1}$.
Here $y_j:=(y_{j,1},\cdots,y_{j,d_j}) \in {\mathbb{R}}^{d_j}$ and $y:=(y_1,\cdots,y_l)$ where $d_j>0$ is a suitably chosen integer.


 
We investigate some partial derivatives of the function defined canonically from the real polynomial $f_j(x_1,\cdots,x_n)-{\Sigma}_{j^{\prime}=1}^{d_j} {y_{j,j^{\prime}}}^2$. For these derivatives, variants are $x_j$ and $y_{j,j^{\prime}}$. 

We first choose a point $(x_0,y_0) \in S$ such that $y_0$ is not the origin.
By the assumption $f_{j}(x_0)>0$ for each $j$. 
Here we consider the partial derivative of the function for each variant $y_{j,j^{\prime}}$.
We introduce our notation on $y_0$ by $y_0:=(y_{0,1},\cdots y_{0,l})$ and $y_{0,j}:=(y_{0,j,1},\cdots,y_{0,j,d_j}) \in {\mathbb{R}}^{d_j}$. 
As a result, we have $2{y_{j,j_0}}=2y_{0,j,j_0} \neq 0$ for some variant $y_{j,j_0}$ where $j^{\prime}=j_0$ in $y_{j,j^{\prime}}$. 

The differential of the restriction of the function defined canonically from the real polynomial $f_j(x)-{||y_j||}^2$ at $(x_0,y_0)$ is not of rank $0$. This is not a singular point of this function defined from $f_j(x_1,\cdots,x_n)-{\Sigma}_{j^{\prime}=1}^{d_j} {y_{j,j^{\prime}}}^2$. We also have the fact that the map into ${\mathbb{R}}^l$ obtained canonically from the $l$ functions is of rank $l$.

Next, we choose a point $(x_{\rm O},y_{\rm O}) \in S$ such that $y_{\rm O}$ is the origin. We can also see that $x_{\rm O} \in S_{j}$ for $j \in J \subset {\mathbb{N}}_l$ and $x_{\rm O} \notin S_{j}$ for $j \notin J$ for a suitably chosen set $J$.

By the assumption, for each real polynomial $f_{j_1}(x)$ with $j_1 \notin J$, $f_{j_1}(x_{\rm O})> 0$. The function defined canonically from the real polynomial $f_{j_0}$ with $j_0 \in J$ has no singular points on the hypersurfaces $S_{j_0}$ by the assumption that $S_{j}$ is non-singular for each general $j$. 
Remember also that the notion of a non-singular real algebraic manifold is via these real polynomials and the naturally defined map with the rank of its differential. Consider the partial derivative of the function defined from the real polynomial $f_{j_0}(x)-{||y_{j_0}||}^2$ with $j_0 \in J$ for each general variant $x_{j}$. For the map into ${\mathbb{R}}^{|J|}$ obtained canonically from the $|J|$ real polynomials, the rank of the differential at the point is $|J|$. This comes from the assumption on the "transversality" for the hypersurfaces $S_j$ in Definition \ref{def:1}. 

By respecting the real polynomials, we consider the canonically obtained map into ${\mathbb{R}}^{l-|J|}$ similarly and we have that the rank of the differential at the point is $l-|J|$. We can see this by considering the partial derivative by some variant $y_{j_1,j^{\prime}}$ with $j_1 \notin J$. This comes from an argument of the case of $(x_0,y_0)$. We have a result as in the case $(x_0,y_0)$.

By the assumption and the definition, for each point ${x_0}^{\prime}$ in ${\mathbb{R}}^n-\overline{D}$ sufficiently close to $D$, the set $S_{{x_0}^{\prime}}:= \{({x_0}^{\prime},y) \in {\mathbb{R}}^n \times {\mathbb{R}}^{m-n+1} \subset {\mathbb{R}}^n \times {\mathbb{R}}^{m-n+1}={\mathbb{R}}^{m+1} \mid f_j(x_1,\cdots,x_n)-||y_j||^2=0, j \in {\mathbb{N}_l}\}$ is empty. From this with the implicit function theorem for example, the original set $S$ is a smooth closed and connected manifold. We can easily see that it is also a connected component of the real algebraic set defined by the $l$ real polynomials. 

By the structures of our sets, we can put $M:=S$ and define $f_D:M \rightarrow {\mathbb{R}}^n$ as the canonical projection to ${\mathbb{R}}^n$. This completes the proof.
	
\end{proof}

 \subsection{A proof of Main Theorems.}

\begin{MainThm}
\label{mthm:2}	
In Main Theorem \ref{mthm:1}, we can take $D$ as in Definition \ref{def:1} and a map $f_D$ as in Proposition \ref{prop:1}. We abuse the notation. We can have the set $D$ and the map $f_D:M \rightarrow {\mathbb{R}}^n$ enjoying the following properties. 
\begin{enumerate}
	\item $\overline{D}-D$ is a union of finitely many $2$-dimensional smooth connected submanifolds of connected components of the real algebraic sets defined by single real polynomials of degrees at most $2$. Let $\{S_j\}_{j=1}^l$ denote the family of these connected components of hypersurfaces as in Definition \ref{def:1}.
\item Each connected component $S_j$ of these hypersurfaces is the product of a connected component of a hyperbola and a copy of the {\rm (}$n-1${\rm )}-dimensional affine space ${\mathbb{R}}^{n-1}$.
\end{enumerate}
\end{MainThm}

\begin{proof}[A proof of Main Theorems]
We investigate several cases. In each case, we choose an NC semi-algebraic set $D$ in a suitable way. 
Of course we respect the four conditions (\ref{def:1.1}), (\ref{def:1.2}), (\ref{def:1.3}) and (\ref{def:1.4}) in these cases.
We explain about the condition (\ref{def:1.5}) after these arguments as a short
remark.

\ \\  
Case 1. The case $l=2$ with $l_{{\mathbb{N}}_{l-1}}({\mathbb{N}}_{l-1})=\{0\}$. \\
	\indent We can take $D_0$ as the interior of the bounded connected closed set surrounded by two connected subsets $C_{{\rm H}_{+},a_{1,1},b_{1,1},c_{1,1}}$ and $C_{{\rm H}_{-},a_{1,2},b_{1,2},c_{1,2}}$ enjoying the following properties.
\begin{itemize}
\item $c_{1,1}, c_{1,2}>0$.
\item The set $C_{{\rm H}_{+},a_{1,1},b_{1,1},c_{1,1}} \bigcap C_{{\rm H}_{-},a_{1,2},b_{1,2},c_{1,2}}$ is represented by $\{(t_{1},s_{1,1}),(t_{2},s_{1,2})\}$ with $s_{1,1}>s_{1,2}$.
\item We can take $S_1:=C_{{\rm H}_{+},a_{1,1},b_{1,1},c_{1,1}}$ and $S_2:=C_{{\rm H}_{-},a_{1,2},b_{1,2},c_{1,2}}$ to apply Proposition \ref{prop:1} by putting $D:=D_0$ to complete the proof in this case.
\end{itemize}
Here $D$ is a so-called convex set. \\
\ \\
	\noindent Case 2. The case $l \geq 3$ with $l_{{\mathbb{N}}_{l-1}}({\mathbb{N}}_{l-1})=\{0\}$. \\
\indent 
First we introduce a subset in ${\mathbb{R}}^2$ important in our construction.

For each $t_j$ satisfying $2 \leq j \leq l-1$, we can suitably choose two connected subsets $C_{{\rm H}_{+},a_{2,j,1},b_{2,j,1},c_{2,j,1}}$ and $C_{{\rm H}_{+},a_{2,j,2},b_{2,j,2},c_{2,j,2}}$ and another two connected subsets $C_{{\rm H}_{-},a_{2,j,3},b_{2,j,3},c_{2,j,3}}$ and $C_{{\rm H}_{-},a_{2,j,4},b_{2,j,4},c_{2,j,4}}$ enjoying the following properties.
\begin{itemize}
\item $a_{2,j,1}<t_j<a_{2,j,2}$. $a_{2,j,3}>t_j>a_{2,j,4}$.
\item $b_{2,j,1}=b_{2,j,2}=b_{2,j,3}=b_{2,j,4}=0$.
\item $c_{2,j,1}, c_{2,j,3}>0$ and $c_{2,j,2}, c_{2,j,4}<0$.
\item The set $C_{{\rm H}_{+},a_{2,j,1},b_{2,j,1},c_{2,j,1}} \bigcap C_{{\rm H}_{+},a_{2,j,2},b_{2,j,2},c_{2,j,2}}$ is represented by $\{(t_j,s_{2,j,1})\}$ for a suitably chosen real number $s_{2,j,1}>0$.
\item The set $C_{{\rm H}_{-},a_{2,j,3},b_{2,j,3},c_{2,j,3}} \bigcap C_{{\rm H}_{+},a_{2,j,4},b_{2,j,4},c_{2,j,4}}$ is represented by $\{(t_j,s_{2,j,2})\}$ for a suitably chosen real number $s_{2,j,2}<0$.
\item The set $C_{{\rm H}_{+},a_{2,j,1},b_{2,j,1},c_{2,j,1}} \bigcup C_{{\rm H}_{+},a_{2,j,2},b_{2,j,2},c_{2,j,2}}$ is below the set \\ $C_{{\rm H}_{+},a_{2,j,3},b_{2,j,3},c_{2,j,3}} \bigcup C_{{\rm H}_{+},a_{2,j,4},b_{2,j,4},c_{2,j,4}}$.
\item The set $C_{{\rm H}_{-},a_{2,j,1},b_{2,j,1},c_{2,j,1}} \bigcup C_{{\rm H}_{-},a_{2,j,2},b_{2,j,2},c_{2,j,2}}$ is below the set \\ $C_{{\rm H}_{-},a_{2,j,3},b_{2,j,3},c_{2,j,3}} \bigcup C_{{\rm H}_{-},a_{2,j,4},b_{2,j,4},c_{2,j,4}}$.
\item We can choose $\{S_{j^{\prime}}\}_{j^{\prime}=1}^4$ as the family of these four connected curves in Definition \ref{def:1}. In other words, we put $S_1:=C_{{\rm H}_{+},a_{2,j,1},b_{2,j,1},c_{2,j,1}}$ and $S_2:=C_{{\rm H}_{+},a_{2,j,2},b_{2,j,2},c_{2,j,2}}$ for the former two and $S_3:=C_{{\rm H}_{-},a_{2,j,3},b_{2,j,3},c_{2,j,3}}$ and $S_4:=C_{{\rm H}_{-},a_{2,j,4},b_{2,j,4},c_{2,j,4}}$ for the latter two. After that we define $D_{t_1,t_j,t_l} \subset {\mathbb{R}}^2$ as the connected open set surrounded by these four curves.
\end{itemize}


\indent We define a desired set $D$. 
We define $D$ as a subset of ${\mathbb{R}}^2 \times {\mathbb{R}}^{l-2}={\mathbb{R}}^{l}$ and see that this is our desired set. 
A point $(x_1,x_2,\cdots,x_{l}) \in {\mathbb{R}}^{l}$ is defined to be a point in $D$ if and only if the following two hold.
\begin{itemize}
\item $(x_1,x_2) \in D_0$ where $D_0$ is a desired set in Case 1 where the pair "$(t_1,t_2)$" there is replaced by $(t_1,t_l)$.
\item For each integer $2 \leq j \leq l-1$, $(x_1,x_{j+1}) \in D_{t_1,t_j,t_l}$.
\end{itemize}
By this rule, we can define a desired set $D$. We can complete our proof in this case similarly. \\


\ \\
\noindent Case 3. A general case for $l \geq 3$. \\
\indent We define another important subset in ${\mathbb{R}}^2$. 

For each $t_j$ satisfying $1 \leq j \leq l-1$, we can suitably choose two connected subsets $C_{{\rm H}_{+},a_{3,j,1},b_{3,j,1},c_{3,j,1}}$ and $C_{{\rm H}_{+},a_{3,j,2},b_{3,j,2},c_{3,j,2}}$ and another two connected subsets $C_{{\rm H}_{-},a_{3,j,3},b_{3,j,3},c_{3,j,3}}$ and $C_{{\rm H}_{-},a_{3,j,4},b_{3,j,4},c_{3,j,4}}$ enjoying the following properties.
\begin{itemize}
\item $a_{3,j,1}=t_{j+1}>a_{3,j,2}=t_{j}$. $a_{3,j,3}>t_j>a_{3,j,4}$.
\item $b_{3,j,1}=b_{3,j,2}=b_{3,j,3}=b_{3,j,4}=0$.
\item $c_{3,j,1}, c_{3,j,3}>0$ and $c_{3,j,2}, c_{3,j,4}<0$.
\item The set $C_{{\rm H}_{+},a_{3,j,1},b_{3,j,1},c_{3,j,1}} \bigcap C_{{\rm H}_{+},a_{3,j,2},b_{3,j,2},c_{3,j,2}}$ is empty.
\item The set $C_{{\rm H}_{-},a_{3,j,3},b_{3,j,3},c_{3,j,3}} \bigcap C_{{\rm H}_{-},a_{3,j,4},b_{3,j,4},c_{3,j,4}}$ is represented by $\{(t_j,s_{3,j})\}$ for a suitably chosen real number $s_{3,j}<0$.
\item The set $C_{{\rm H}_{+},a_{3,j,1},b_{3,j,1},c_{3,j,1}} \bigcup C_{{\rm H}_{+},a_{3,j,2},b_{3,j,2},c_{3,j,2}}$ is below the set \\ $C_{{\rm H}_{+},a_{3,j,3},b_{3,j,3},c_{3,j,3}} \bigcup C_{{\rm H}_{+},a_{3,j,4},b_{3,j,4},c_{3,j,4}}$.
\item The set $C_{{\rm H}_{-},a_{3,j,1},b_{3,j,1},c_{3,j,1}} \bigcup C_{{\rm H}_{-},a_{3,j,2},b_{3,j,2},c_{3,j,2}}$ is below the set \\ $C_{{\rm H}_{-},a_{3,j,3},b_{3,j,3},c_{3,j,3}} \bigcup C_{{\rm H}_{-},a_{3,j,4},b_{3,j,4},c_{3,j,4}}$.
\item We can choose $\{S_{j^{\prime}}\}_{j^{\prime}=1}^4$ as the family of these four connected curves in Definition \ref{def:1}. In other words, we put $S_1:=C_{{\rm H}_{+},a_{3,j,1},b_{3,j,1},c_{3,j,1}}$ and $S_2:=C_{{\rm H}_{+},a_{3,j,2},b_{3,j,2},c_{3,j,2}}$ for the former two and $S_3:=C_{{\rm H}_{-},a_{3,j,3},b_{3,j,3},c_{3,j,3}}$ and $S_4:=C_{{\rm H}_{-},a_{3,j,4},b_{3,j,4},c_{3,j,4}}$ for the latter two. After that we define $D_{t_1,t_j,t_{j+1},t_l} \subset {\mathbb{R}}^2$ as the connected open set surrounded by these four curves.
\end{itemize}
\indent We define a desired set $D$. 
We define $D$ as a subset of ${\mathbb{R}}^2 \times {\mathbb{R}}^{l-1}={\mathbb{R}}^{l+1}$ and see that this is our desired set. 
A point $(x_1,x_2,\cdots,x_{l+1}) \in {\mathbb{R}}^{l+1}$ is defined to be a point in $D$ if and only if the following two hold.
\begin{itemize}
\item $(x_1,x_2) \in D_0$ where $D_0$ is a desired set in Case 1 where the pair "$(t_1,t_2)$" there is replaced by the pair $(t_1,t_l)$.
\item $(x_1,x_3) \in D_{t_1,t_2,t_l}$ if $l_{{\mathbb{N}}_{l-1}}(1)=0$ and $(x_1,x_3) \in D_{t_1,t_1,t_2,t_l}$ if $l_{{\mathbb{N}}_{l-1}}(1)=1$.
\item For each integer $2 \leq j \leq l-1$, $(x_1,x_{j+2}) \in D_{t_1,t_j,t_l}$ if $l_{{\mathbb{N}}_{l-1}}(j)=0$ and $(x_1,x_{j+2}) \in D_{t_1,t_j,t_{j+1},t_l}$ if $l_{{\mathbb{N}}_{l-1}}(j)=1$.
\end{itemize}
By this rule, we can define a desired set $D$. We can complete our proof in this case similarly. \\

\ \\
\indent We give some comments on the proof.

By considering suitable mutually independent tangent vectors to the intersections of distinct hypersurfaces here or mutually independent normal vectors to the mutually distinct hypersurfaces, we can check the condition on the transversality or the condition (\ref{def:1.5}) of Definition \ref{def:1}.

Similar exposition via normal vectors is also presented shortly in Remark 2 of \cite{kitazawa7} and this is
a fundamental argument.

More precisely, the following shows a key ingredient in knowing this, for example. Every hypersurface $S_j$ in Case 2 or Case 3, where we abuse the notation of Definition \ref{def:1} and Proposition \ref{prop:1},
is represented as the product of a connected component of a hyperbola in some copy of the real affine space ${\mathbb{R}}^2$ embedded in ${\mathbb{R}}^l$ or ${\mathbb{R}}^{l+1}$
straightly and a copy of the real affine space ${\mathbb{R}}^{l-2}$ or ${\mathbb{R}}^{l-1}$ embedded vertically. 
We can define the copy of the affine space ${\mathbb{R}}^2$ uniquely in a suitable way and we do.
For our distinct hypersurfaces, the number of the canonically corresponding 
connected components of hyperbolas in some copy of the real affine space ${\mathbb{R}}^2$ embedded in ${\mathbb{R}}^l$ or ${\mathbb{R}}^{l+1}$
straightly before is always at most $2$. Any two distinct connected curves here intersect satisfying the transversality of course and the number of the intersection is at most $1$.

We also give a short remark on singular points of the resulting function $f:={\pi}_{n,1} \circ f_D$ where we abuse the notation. In short, at least one singular point of $f$ appears in the preimage ${f_D}^{-1}(p)$ of a point $p \in \overline{D}$ if and only if $p$ is contained in at least two distinct hypersurfaces in the family $\{S_j\}$. 

This completes the proof.
	\end{proof}
\subsection{Future problems.}
As our future problems, we propose new and natural problems explicitly.

\begin{Prob}
\label{prob:2}
For given data as in Problem \ref{prob:1} (with Main Theorems \ref{mthm:1} and \ref{mthm:2}), can you formulate a simplest or most natural smooth real algebraic function $f:M \rightarrow \mathbb{R}$ and the manifold $M$? It is also important to respect $f_D:M \rightarrow {\mathbb{R}^n}$ enjoying the relation $f:={\pi}_{n,1} \circ f_D$.
\end{Prob}
Our previous example in \cite{kitazawa7} seems to be most natural. 
Our new examples may be simpler in choosing each hypersurface as the product of a connected component of a hyperbola and a copy of the $1$-dimensional real affine space $\mathbb{R}$.

We consider more explicit cases and cases seeming to be simplest and most fundamental. For example, in the case $l=2$, it seems to be most natural to obtain the function $f:={\pi}_{n,1} \circ f_D$ enjoying the following properties.
\begin{itemize}
\item There exists some positive integer $m>0$.
\item For $m$ before and any integer $n$ satisfying $2 \leq n \leq m$, we can consider a canonical projection of a unit sphere or the product map of a canonical projection of a unit sphere into the real affine space $\mathbb{R}$ and the identity map on the real affine space $\mathbb{R}$, denoted by $f_D:M_0 \rightarrow {\mathbb{R}}^n$, and we have some diffeomorphism $\Phi:M \rightarrow M_0$ enjoying the relation $f_0 \circ \Phi=f_D$. 
Here, a diffeomorphism means a smooth map with no singular points which is also a homeomorphism. As $\Phi:M \rightarrow M_0$, it may be also natural to choose $\Phi$ as a real algebraic diffeomorphism.
\end{itemize}

As another example, we explain about the case $l=4$ and $l_{{\mathbb{N}}_3}(\{1,2,3\})=\{0\}$. In this case, it seems to be most natural to obtain a map $f_D:M \rightarrow {\mathbb{R}}^n$ with the function $f:={\pi}_{n,1} \circ f_D$ for some integer $n \geq 3$ enjoying the following two properties.
\begin{itemize}
\item There exists some positive integer $k>0$. We define $m:=k+n-1$ in our Main Theorems.
\item For $n$, $k$ and $m$ before, we consider the product map of the following two natural maps.
\begin{itemize}
\item A canonical projection of a unit sphere $S^{k}$ into the real affine space $\mathbb{R}$.
\item The identity map on a copy of the unit sphere $S^{n-1}$.
\end{itemize}

We also consider the composition with a natural embedding into ${\mathbb{R}}^n$. We have a map represented in this way and denoted by the form $f_0:M_0:=S^{k} \times S^{n-1} \rightarrow \mathbb{R} \times S^{n-1} \rightarrow {\mathbb{R}}^n$. For some map obtained in this way, we have some diffeomorphism $\Phi:M \rightarrow M_0$ enjoying the relation $f_0 \circ \Phi=f_D$. 
Here, a diffeomorphism means a smooth map with no singular points which is also a homeomorphism as before. It may be also natural to choose $\Phi$ as a real algebraic diffeomorphism as before.
\end{itemize}
According to our construction, types of singular points of $f_D$ seem to be more complicated than general cases. This exposition comes from the viewpoint of singularity theory. Here we do not explain about singularity theory of differentiable maps precisely. For related systematic theory, see \cite{golubitskyguillemin} for example.
\begin{Prob}
\label{prob:3}
Can we formulate an explicit algorithm or an explicit rule to obtain the maps and the manifolds as just before?
\end{Prob}

This is more explicit problem than Problem \ref{prob:2}.

These problems also seem to be natural problems in smooth cases. We may see that some explicit answers are already given in \cite{masumotosaeki, saeki2} with papers \cite{kitazawa1, kitazawa2} by the author for example in the smooth cases. It is also important that \cite{saeki2} constructs smooth maps which are not real analytic. \cite{kitazawa2} also constructs such maps. They also imply that constructing smooth real algebraic functions whose Reeb graphs are given graphs and whose preimages are of prescribed types is difficult in general situations.

As a bit different viewpoint on this, \cite{kitazawa4, kitazawa8} construct such maps for smoothly immersed $n$-dimensional manifolds whose dimensions are same as the dimensions $n$ of the manifolds of the targets. They are so-called {\it special generic} maps. Special generic maps are, in short, naturally generalized versions of Morse functions with exactly two singular points on spheres in Reeb's sphere theorem and canonical projections of unit spheres. Maps constructed in Proposition \ref{prop:1} are topologically regarded as special generic maps like ones obtained in this way where we consider cases the family of hypersurfaces consisting of exactly one hypersurface. Here "topologically" is for the existence of homeomorphisms $\Phi:M \rightarrow M_0$ enjoying the properties discussed just before including the property that the two manifolds $M$ and $M_0$ are diffeomorphic. Of course two manifolds are defined to be diffeomorphic if and only if there exists a diffeomorphism between these two manifolds. 

For special generic maps, see also \cite{saeki1} for example. This is a pioneering systematic study on algebraic topological or differential topological properties of special generic maps and the manifolds.
The class of special generic maps is
important in
algebraic topology and differential topology of manifolds. For example, such maps restrict the topologies of the manifolds and the differentiable structures of the spheres and general manifolds
 of the domains strongly.

Our problems may be variants of Problem 2.3 of \cite{saeki3}. The problem is a natural problem in singularity theory and applications to differential topology of manifolds. This asks us to give most standard or simplest maps (whose codimensions are not positive) on given smooth (closed) manifolds.

\end{document}